\documentclass[11pt,a4paper]{article}

\usepackage{amsmath}
\usepackage{amssymb}
\usepackage{amsfonts}

\usepackage{amsthm}
\usepackage{mathtools} 
\usepackage{enumerate}
\usepackage{xspace} 
\usepackage{graphics} 
\usepackage{graphicx}
\usepackage{pgfpages}
\usepackage{url}
\usepackage{booktabs}
\usepackage{verbatim}
\usepackage{hyperref}
\usepackage{breakurl}
\usepackage[margin=1.43in]{geometry}
\newcommand{\tluste}[1]{\text{\mathversion{bold}$\mathbf #1 $}}
\newcommand{\vr}[1]{{{#1}}}

\newcommand{\mace}[1]{{{#1}}}
\newcommand{\mna}[1]{{\mathcal{#1}}}

\newcommand{\omace}[1]{\mbox{$\overline{\mace{#1}}$}} 
\newcommand{\umace}[1]{\mbox{$\underline{\mace{#1}}$}} 
\newcommand{\imace}[1]{\tluste{#1}} 		
 
\newcommand{\smace}[1]{\tluste{#1}{}^S} 
\def\Mid#1{{#1}^c}		
\def\Rad#1{{#1}^\Delta}		
\def\Mag#1{\mathop{\mathrm{mag}}#1}	
\def\Mig#1{\mathop{\mathrm{mig}}#1}	
\newcommand{\ovr}[1]{\mbox{$\overline{\vr{#1}}$}} 
\newcommand{\uvr}[1]{\mbox{$\underline{\vr{#1}}$}}

\newcommand{\onum}[1]{\mbox{$\overline{{#1}}$}} 
\newcommand{\unum}[1]{\mbox{$\underline{{#1}}$}}

\newcommand{\ivr}[1]{\mbox{$\tluste{#1}$}} 

\newcommand{\inum}[1]{\mbox{$\tluste{#1}$}} 
\newcommand{\R}[0]{{\mathbb{R}}}

\newcommand{\IR}[0]{{\mathbb{IR}}}

\newcommand{\Ss}[0]{\Sigma}
\newcommand{\Ssp}[0]{\mbox{\large$\Sigma$}_{\mathrm{p}}}
\def\ihull{{\mbox{$\square$}}}

\newcommand{\mmid}[0]{;\,}		

\newcommand{\seznam}[1]{{\{1, \ldots, {#1}\}}}

\def\clqq{``}
\def\crqq{''}
\def\quo#1{\clqq{}#1\crqq{}}  

\newcommand{\minim}[2]{{\min {\{#1; \ {} #2 \}} }}
\newcommand{\maxim}[2]{{\max {\{#1; \ {} #2 \}} }}

\DeclareMathOperator{\sgn}{sgn}	
\DeclareMathOperator{\diag}{diag}	
\DeclareMathOperator{\rank}{rank}
\DeclareMathOperator{\adj}{adj}

\DeclareMathOperator{\rr}{r}

\newtheorem{theorem}{Theorem}

\newtheorem{corollary}{Corollary}
\newtheorem{conjecture}{Conjecture}
\theoremstyle{definition}

\begin{document}

\title{An Overview of Polynomially Computable Characteristics of Special Interval Matrices}

\author{
  Milan Hlad\'{i}k\footnote{
Charles University, Faculty  of  Mathematics  and  Physics,
Department of Applied Mathematics, 
Malostransk\'e n\'am.~25, 11800, Prague, Czech Republic, 
e-mail: \texttt{milan.hladik@matfyz.cz}
}}

\maketitle

\begin{abstract}
It is well known that many problems in interval computation are intractable, which restricts our attempts to solve large problems in reasonable time. This does not mean, however, that all problems are computationally hard. Identifying polynomially solvable classes thus belongs to important current trends. 
The purpose of this paper is to review some of such classes. In particular, we focus on several special interval matrices and investigate their convenient properties. We consider tridiagonal matrices, \{M,H,P,B\}-matrices, inverse M-matrices, inverse nonnegative matrices, nonnegative matrices, totally positive matrices and some others. 
We focus in particular on computing the range of the determinant, eigenvalues, singular values, and selected norms. Whenever possible, we state also formulae for determining the inverse matrix and the hull of the solution set of an interval system of linear equations.
We survey not only the known facts, but we present some new views as well.
\end{abstract}


\section{Introduction}


Many problems in interval computation are computationally hard; see theoretic complexity surveys in \cite{KreLak1998,HorHla2017a}.
Nevertheless, matrices arising in practical problems are not random, but satisfy some special properties and have specific structures. Utilizing such particularities is often very convenient and can make tractable those problems that are hard on general. In this paper, we review such special matrices and easily computable characteristics.

\subsubsection*{General notation}
For a symmetric matrix $A\in\R^{n\times n}$, we denote its eigenvalues as $\lambda_{\max}(A)=\lambda_1(A)\geq\dots\geq\lambda_n(A)=\lambda_{\min}(A)$.
For any matrix $A\in\R^{n\times n}$, we use $\rho(A)$ for the spectral radius, and $\sigma_{\min}(A)$ and $\sigma_{\max}(A)$ the smallest and the largest singular value, respectively. Further, $\diag(z)$ stand for the diagonal matrix with entries $z_1,\dots,z_n$, the symbol $I_n$ is for the identity matrix of size $n$, and $e=(1,\dots,1)^T$ for an all-ones vector of convenient dimension. 
The $i$th row and the $j$th column of a matrix $A$ are denoted by $A_{i*}$ and $A_{*j}$, respectively.
Throughout the text, inequalities between vectors and matrices well as the absolute values and min/max functions are understood entrywise.

The regularity radius \cite{KreLak1998,PolRoh1993} of a nonsingular matrix $A\in\R^{n\times n}$ is the distance to the nearest singular matrix in the Chebyshev norm (componentwise maximum norm) and denoted
$$
\rr(A)\coloneqq
\min\{\delta\geq0\mmid \exists \mbox{ singular }B\in\R^{n\times n}:
  |a_{ij}-b_{ij}|\leq\delta\ \forall i,j\}.
$$
This value can be expressed as $\rr(A)=1/\|A^{-1}\|_{\infty,1}$, where 
$$
\|M\|_{\infty,1}\coloneqq\max_{\|x\|_{\infty}=1}\|Mx\|_1
$$
is the matrix norm induced by the vector $\infty$- and 1-norms.
Computing this norm is, however, an NP-hard problem on the set of symmetric rational M-matrices \cite{Fie2006,Roh2000}. The best known approximation is by means of semidefinite programming \cite{HarHla2016a}.

\subsubsection*{Interval notation}

An interval matrix is defined as
$$
\imace{A}\coloneqq\{A\in\R^{m\times n}\mmid \umace{A}\leq A\leq \omace{A}\},
$$
where $ \umace{A}$ and $\omace{A}$, $\umace{A}\leq\omace{A}$, are given matrices. The corresponding midpoint and the radius matrices are defined respectively as
$$
\Mid{A}\coloneqq\frac{1}{2}(\umace{A}+\omace{A}),\quad
\Rad{A}\coloneqq\frac{1}{2}(\omace{A}-\umace{A}).
$$
The set of all $m\times n$ interval matrices is denoted by $\IR^{m\times n}$, and intervals and interval vectors are considered as special cases of interval matrices. For interval arithmetic, we refer the reader, e.g., to Neumaier \cite{Neu1990}. 
Given $\imace{A}\in\IR^{n\times n}$ with $\Mid{A}$ and $\Rad{A}$ symmetric, we denote by $\smace{A}\coloneqq\{A\in\imace{A}\mmid A=A^T\}$ the corresponding symmetric interval matrix.

For a bounded set $S\subset\R^n$, the interval hull $\ihull S$ is the smallest enclosing interval vector, or more formally,  $\ihull S\coloneqq\cap\{ \ivr{v}\in\IR^n\mmid S\subseteq\ivr{v}\}$ .

Consider an interval system of linear equations $\imace{A}x=\ivr{b}$, where $\imace{A}\in\IR^{m\times n}$ and $\ivr{b}\in\R^m$. Its solutions set $\Ss$ is traditionally defined as the union of all solutions of realizations of interval coefficients, that is
$$
\Ss\coloneqq
\{x\in\R^n\mmid \exists A\in\imace{A},\,\exists b\in\ivr{b}:Ax=b\}.
$$
Consider any matrix property $\mathfrak{P}$. We say that an interval matrix $\imace{A}$ satisfies $\mathfrak{P}$ is every $A\in\imace{A}$ satisfies $\mathfrak{P}$. This applies in particular to regularity (every $A\in\imace{A}$ is nonsingular), positive definiteness, M-matrix property, nonnegativity and others.
Recall that checking whether an interval matrix is regular is a co-NP-hard problem \cite{Fie2006,KreLak1998,PolRoh1993}. 

For a real function $f\colon\R^{m\times n}\to\R$ and an interval matrix $\imace{A}\in\IR^{m\times n}$, the image of the interval matrix under the function is
$$
f(\imace{A})=\{f(A)\mmid A\in\imace{A}\}.
$$
In general, $f(\imace{A})$ needn't be an interval, but it is the case provided $f$ is continuous. Thus, for instance, $\det(\imace{A})$ gives the range of determinant of $\imace{A}$ or $\lambda_{\max}(\smace{A})$ gives the range of the largest eigenvalues of the symmetric interval matrix $\smace{A}$.

\section{Tridiagonal matrices}

Tridiagonal interval matrices have particularly nice properties and some NP-hard problems become polynomial in this class. Let $\imace{T}\in\IR^{n\times n}$ be a tridiagonal interval matrix, that is,  $\imace{T}_{ij}=0$ for $|i-j|>1$. Checking regularity of $\imace{T}$ can be performed in linear time (Bar-On et al. \cite{BarCod1996}).
However, there are still some open problems. Are polynomially solvable the following tasks?
\begin{itemize} 
\item
computing the exact range for the determinant,
\item
tight enclosure of the solution set of an interval linear system $\imace{T}x=\ivr{b}$,
\item
computing the eigenvalue sets of a symmetric tridiagonal interval matrix,
\item
computing  $\|\imace{T}\|_{\infty,1}$.
\end{itemize}

\section{M-matrices and  H-matrices}

Interval M-matrices and H-matrices are particularly convenient in the context of solving interval linear equations $\imace{A}x=\ivr{b}$.
Recall that $A\in\R^{n\times n}$ is an \emph{M-matrix} if $a_{ij}\leq0$ for every $i\not=j$ and $A^{-1}\geq0$. The condition can be equivalently formulated as any of the following conditions
\begin{itemize}
\item
all real eigenvalues are positive,
\item
real parts of all eigenvalues are positive,
\item
there is $v>0$ such that $Av>0$.
\end{itemize}
\cite{HorJoh1991}.
Due to the statement below, interval M-matrices constitute an easily verifiable regular interval matrices.

\begin{theorem}
An interval matrix $\imace{A}\in\IR^{n\times n}$ is an M-matrix if and only if $\umace{A}$ is an M-matrix and $\omace{A}_{ij}\leq0$ for all $i\not=j$.
\end{theorem}

A matrix $A\in\R^{n\times n}$ is called an \emph{H-matrix}, if the so called comparison matrix $\langle A\rangle$ is an M-matrix, where $\langle A\rangle_{ii}=|a_{ii}|$ and $\langle A\rangle_{ij}=-|a_{ij}|$ for $i\not=j$. Speciasublasses of H-matrices were discussed, e.g., in Cvetkovi{\'c} et al.~\cite{CveKos2009}.

Also interval H-matrices are easy to characterize; see Neumaier \cite{Neu1990}. We have that $\imace{A}\in\IR^{n\times n}$ is an H-matrix if and only if $\langle \imace{A}\rangle$ is an M-matrix, where the notion of the comparison matrix is extended to interval matrices as follows
\begin{align*}
\langle \imace{A}\rangle_{ii}
 &=\Mig(\inum{a}_{ii})
  =\minim{|a|}{a\in\inum{a}_{ii}},\\
\langle \imace{A}\rangle_{ij}
 &=-\Mag(\inum{a}_{ii})
 =-\maxim{|a|}{a\in\inum{a}_{ij}},\quad i\not=j.
\end{align*}
Each diagonally dominant matrix is an H-matrix. So we do not investigate diagonally dominant matrices in particular since what we show for H-matrices holds for diagonally dominant matrices as well.

Each M-matrix is also an H-matrix, so the following results apply to both. By Alefeld \cite{Ale1977}, for an H-matrix $\imace{A}$, the interval Gaussian elimination can be carried out without any pivoting and does not fail. Moreover, for any H-matrix $\imace{A}$ we always find an LU decomposition \cite{Ale1977}. That is, 
there are lower and upper triangular interval matrices $\imace{L},\imace{U}\in\IR^{n\times n}$ such that the diagonal of $\imace{L}$ consists of ones, and $\imace{A}\subseteq\imace{L}\imace{U}$.

Provided that $\imace{A}$ is an M-matrix, and one of $0\in\ivr{b}$, $\uvr{b}\geq0$, or $\ovr{b}\leq0$ holds true, then the interval Gaussian elimination yields the interval hull of the solution set, i.e., $\ihull\Ss$; see \cite{BarNud1974,Bee1974} and Section~\ref{sInvNonneg} for a more general result. 
For a general H-matrix, this needn't be true, however, for any H-matrix $\imace{A}$,  interval hull of the solution set is polynomially solvable by the so called Hansen--Bliek--Rohn--Ning--Kearfott method; see, e.g., \cite{Fie2006,NinKea1997,Neu1999}.


A link between regularity and H-matrix property was given by Neumaier \cite[Prop.~4.1.7]{Neu1990}.

\begin{theorem}\label{thmHmatMidReg}
Let $\Mid{A}$ be an M-matrix. Then $\imace{A}$ is regular if and only if it is an H-matrix.
\end{theorem}

Notice that the assumption cannot be weakened to the assumption that $\Mid{A}$ is an H-matrix. For example, the interval matrix
$$
\imace{A}=\begin{pmatrix}[0,10]&1\\-1&10\end{pmatrix}
$$
is regular and its midpoint is an H-matrix. However, $\imace{A}$ itself is not an H-matrix, failing for the realization when the top left entry vanishes.

As a consequence, we have a result related to positive definiteness.
Checking positive definiteness of interval matrices is co-NP-hard \cite{KreLak1998,Roh1994}, so polynomial recognizable sub-classes are of interest.

\begin{theorem}
Let $\imace{A}\in\IR^{n\times n}$ be an H-matrix and $\Mid{A}$ positive definite. Then $\imace{A}$ is positive definite.
\end{theorem}

\begin{proof}
By \cite{KreLak1998,Roh1994b}, positive definiteness of $\Mid{A}$ and regularity of $\imace{A}$ implies positive definiteness of  $\imace{A}$.
\end{proof}

\begin{theorem}
Let $\Mid{A}$ be a (symmetric) positive definite M-matrix. Then $\imace{A}$ is positive definite if and only if it is an H-matrix.
\end{theorem}

\begin{proof}
By \cite{KreLak1998,Roh1994b}, under the assumption of positive definiteness of $\Mid{A}$, we have that $\imace{A}$ is positive definite if and only if it is regular, which is equivalent to H-matrix property by Theorem~\ref{thmHmatMidReg}.
\end{proof}

\begin{theorem}\label{thmMmatDet}
Let $\imace{A}\in\IR^{n\times n}$ be an M-matrix. Then $\det(\imace{A})=[\det(\umace{A}),\det(\omace{A})]$.
\end{theorem}

\begin{proof}
The derivative of the determinant $\det(A)$ is $\det(A)A^{-T}$. For an M-matrix both the determinant and the inverse are nonnegative, so the determinant is a nondecreasing function in each component.
\end{proof}

Since each M-matrix is inverse nonnegative, Theorems~\ref{thmInvNonnegSvdMin} and~\ref{thmInvNonnegDet} from Section~\ref{sInvNonneg} below are valid also for interval M-matrices.

\section{Inverse nonnegative matrices}\label{sInvNonneg}

Besides the generalization to H-matrices, M-matrices can also be extended to \emph{inverse nonnegative matrices}, that is, matrices $A\in\R^{n\times n}$ such that $A^{-1}\geq0$. Interval inverse nonnegativity is still easy to characterize just by reduction to two point matrices $\umace{A}$ and $\omace{A}$ only; see Kuttler \cite{Kut1971}.

\begin{theorem}
An interval matrix $\imace{A}\in\IR^{n\times n}$ is inverse nonnegative if and only if $\umace{A}^{-1}\geq0$ and $\omace{A}^{-1}\geq0$.
\end{theorem}

For inverse nonnegative matrices we can easily determine the range of their inverses. The theorem below says that $\ihull\{A^{-1}\mmid A\in\imace{A}\}=[\omace{A}^{-1},\umace{A}^{-1}]$.

\begin{theorem}\label{thmInvNonnegAinvBounds}
If $\imace{A}$ is inverse nonnegative, then $\omace{A}^{-1}\leq A^{-1}\leq \umace{A}^{-1}$ for every $A\in\imace{A}$.
\end{theorem}

When an interval matrix $\imace{A}\in\IR^{n\times n}$ is inverse nonnegative, then interval systems $\imace{A}x=\ivr{b}$ are efficiently solvable. The interval hull of the solution se reads
\begin{itemize}
\item
$\ihull\Ss=[\omace{A}^{-1}\uvr{b},\umace{A}^{-1}\ovr{b}]$
 when $\uvr{b}\geq0$,
\item
$\ihull\Ss=[\umace{A}^{-1}\uvr{b},\omace{A}^{-1}\ovr{b}]$
 when $\ovr{b}\leq0$,
\item
$\ihull\Ss=[\umace{A}^{-1}\uvr{b},\umace{A}^{-1}\ovr{b}]$
 when $0\in\ivr{b}$.
\end{itemize}
In the other cases, $\ihull\Ss$ is still polynomially computable, but has no such an explicit formulation; see Neumaier \cite{Neu1990}.

For symmetric inverse nonnegative matrices we have also a simple formula for its smallest eigenvalue. Notice that for the largest eigenvalue an analogy is not valid in general.

\begin{theorem}\label{thmInvNonnegSymLmin}
Let $\imace{A}$ be inverse nonnegative and both $\Rad{A}$ and $\Mid{A}$ symmetric. Then $\lambda_{\min}(\smace{A})=[\lambda_{\min}(\umace{A}),\lambda_{\min}(\omace{A})]$.
\end{theorem}

\begin{proof}
Let $A\in\smace{A}$. Then by the Perron theorem and theory of nonnegative matrices, 
$\lambda_{\min}(A)
=\lambda^{-1}_{\max}(A^{-1})
\geq\lambda^{-1}_{\max}(\umace{A}^{-1})
=\lambda_{\min}(\umace{A}),
$
and similarly for the upper bound.
\end{proof}

Analogously, we obtain:

\begin{theorem}\label{thmInvNonnegSvdMin}
If $\imace{A}$ is inverse nonnegative, then $\sigma_{\min}(\imace{A})=[\sigma_{\min}(\umace{A}),\sigma_{\min}(\omace{A})]$.
\end{theorem}

\begin{theorem}\label{thmInvNonnegDet}
If $\imace{A}$ is inverse nonnegative, then $\det(\imace{A})=[\min(\mna{D}),\max(\mna{D})]$, where $\mna{D}=\{\det(\umace{A}),\det(\omace{A})\}$.
\end{theorem}

\begin{proof}
Analogously to the proof of Theorem~\ref{thmMmatDet} we use that the derivative of the determinant $\det(A)$ is $\det(A)A^{-T}$. The determinant must have a constant sign, and $A^{-T}\geq0$, so the minimal and maximal determinants are attained for $\umace{A}$ or $\omace{A}$.
\end{proof}

The above theorem can simply be extended to sign stable matrices, which are those interval matrices $\imace{A}\in\IR^{n\times n}$ satisfying $|A^{-1}|>0$; see Rohn and Farhadsefat \cite{RohFar2011}. The signs of the entries say if the determinant is nonincreasing or nondecreasing. Therefore, the left/right endpoint of $\det(\imace{A})$ is attained for a matrix $A\in\imace{A}$ defined as $a_{ij}=\unum{a}_{ij}$ if $(A^{-1})_{ij}\geq0$ and $a_{ij}=\onum{a}_{ij}$ otherwise. 

For the regularity radius, we have:

\begin{theorem}\label{thmInvNonnegRr}
If $\imace{A}$ is inverse nonnegative, then $\rr(\imace{A})=[\rr(\umace{A}),\rr(\omace{A})]$.
\end{theorem}

\begin{proof}
Let $A\in\imace{A}$. By Theorem~\ref{thmInvNonnegAinvBounds}, $\rr(A)=1/\|A^{-1}\|_{\infty,1}\leq1/\|\omace{A}^{-1}\|_{\infty,1}=\rr(\omace{A})$, and similarly from below.
\end{proof}

\section{Totally positive matrices}

A matrix $A\in\R^{n\times n}$ is totally positive if the determinants of all submatrices are positive. Despite the definition, checking this property is a polynomial problem; see Fallat and Johnson \cite{FalJoh2011}.

Let $\imace{A}\in\IR^{n\times n}$. First we show a correspondence between total positivity of $\imace{A}$ and inverse nonnegativity. Denote $s\coloneqq(1,-1,1,-1,\dots)^T$ of a convenient length.

\begin{theorem}\label{thmTpInvNonneg}
If $\imace{A}$ is totally positive, then $\diag(s)\imace{A}\diag(s)$ is inverse nonnegative.
\end{theorem}

\begin{proof}
The inverse of $A$ can be expressed as $A^{-1}=\det(A)^{-1}\adj(A)$, where
the entries of the adjugate matrix are defined as $\adj(A)_{ij}=(-1)^{i+j}\det(A^{ji})$, and $A^{ji}$ arises from $A$ by removing the $j$th row and the $i$th column. Thus $\imace{A}$ is inverse sign stable corresponding to the checkerboard order and therefore $\diag(s)\imace{A}\diag(s)$ is inverse nonnegative.
\end{proof}

From the above theorem, we can easily derive many useful properties of totally positive interval matrices based on the results presented in Section~\ref{sInvNonneg}.

Also total positivity of an interval matrix $\imace{A}\in\IR^{n\times n}$ can also be verified in polynomial time just by reducing the problem to two vertex matrices defined by the checkerboard order. Define $\downarrow A,\uparrow A\in\imace{A}$ as follows
\begin{align*}
\downarrow A\coloneqq\Mid{A}-\diag(s)\Rad{A}\diag(s),\quad
\uparrow A\coloneqq\Mid{A}+\diag(s)\Rad{A}\diag(s).
\end{align*}
In relation to Theorem~\ref{thmTpInvNonneg}, these matrices can also be expressed as
\begin{align*}
\downarrow A
&=\diag(s)(\umace{\diag(s)\imace{A}\diag(s)})\diag(s),\\
\uparrow A
&=\diag(s)(\omace{\diag(s)\imace{A}\diag(s)})\diag(s).
\end{align*}
Then we have all ingredients to state the result by Garloff \cite{Gar1982}:

\begin{theorem}
$\imace{A}$ is totally positive if and only if $\downarrow A$ and $\uparrow  A$ are totally positive.
\end{theorem}

As consequences, we obtain the following properties.

\begin{corollary}
If $\imace{A}$ is totally positive, then $\sigma_{\min}(\imace{A})=[\sigma_{\min}(\downarrow A),\sigma_{\min}(\uparrow A)]$ and $\sigma_{\max}(\imace{A})=[\sigma_{\max}(\umace{A}),\sigma_{\max}(\omace{A})]$.
\end{corollary}

\begin{proof}
The formula for $\sigma_{\min}(\imace{A})$ follows from Theorems~\ref{thmInvNonnegSvdMin} and~\ref{thmTpInvNonneg}.
The formula for $\sigma_{\max}(\imace{A})$ will be shown in Theorem~\ref{thmNonnegSvd} under weaker assumptions; notice that $\imace{A}$ here is componentwisely nonnegative.
\end{proof}

\begin{corollary}
If $\imace{A}$ is totally positive, then $\det(\imace{A})=[\min(\mna{D}),\max(\mna{D})]$, where $\mna{D}=\{\det(\downarrow  A),\det(\uparrow A)\}$.
\end{corollary}

\begin{proof}
It follows from Theorems~\ref{thmInvNonnegDet} and~\ref{thmTpInvNonneg}.
\end{proof}

\begin{corollary}
If $\imace{A}$ is totally positive, then $\rr(\imace{A})=[\rr(\downarrow  A),\rr(\uparrow A)]$.
\end{corollary}

\begin{proof}
It follows from Theorems~\ref{thmInvNonnegRr} and~\ref{thmTpInvNonneg}.
\end{proof}

Totally positive matrices have distinct positive eigenvalues $\lambda_1,>\dots>\lambda_n>0$ the properties of which enable us to compute the eigenvalue ranges of interval matrices.

\begin{theorem}
If $\imace{A}$ is totally positive, then $\lambda_n(\imace{A})=[\lambda_n(\downarrow A),\lambda_n(\uparrow A)]$ and $\lambda_1(\imace{A})=[\lambda_1(\umace{A}),\lambda_1(\omace{A})]$.
\end{theorem}

\begin{proof}
Let $A\in\imace{A}$ and let $x,y$ be the right and left eigenvectors of $A$ corresponding to the smallest eigenvalue $\lambda_n(A)$ and normalized such that $x^Ty=1$. By Fallat \& Johnson \cite{FalJoh2011}, the signs of both vectors $x$ and $y$ alternate, so we can assume that both have the sign vector given by $s$ defined above, that is, $\sgn(x)=\sgn(y)=s$. The derivative of $\lambda_n(A)$ with respect to $a_{ij}$ is $x_iy_j$, so the maximum is attained for $\Mid{a}_{ij}+s_is_j\Rad{a}_{ij}=(\uparrow A)_{ij}$ and similarly for the minimum.

The second formula follows from Perron theory of eigenvalues of nonnegative matrices. For each $A\in\imace{A}$ we have $\lambda_1(A)=\rho(A)\leq\rho(\omace{A})=\lambda_1(\omace{A})$, and similarly of the lower bound.
\end{proof}

Even more, we can easily compute eigenvalue sets $\lambda_i(\imace{A})$ for any other $i\in\seznam{n}$. By Fallat \& Johnson \cite{FalJoh2011}, the signs of both left and right eigenvectors corresponding to $\lambda_i(A)$ are constant for every $A\in\imace{A}$ (eigenvalues of principal submatrices of size $n-1$ strictly interlace eigenvalues of $A$, so no eigenvector has a zero entry). Therefore, we can proceed as follows. Let $x$ and $y$, $x^Ty=1$, be the eigenvectors corresponding to $\lambda_i(\Mid{A})$. Then $\lambda_i(\imace{A})=[\lambda_i(A^1),\lambda_i(A^2)]$, where $A^1$ and $A^2$ are defined as
\begin{align*}
A^1&=\Mid{A}-\diag(\sgn(x))\Rad{A}\diag(\sgn(y)),\\
A^2&=\Mid{A}+\diag(\sgn(x))\Rad{A}\diag(\sgn(y)).
\end{align*}

Consider now an interval system $\imace{A}x=\ivr{b}$ with $\imace{A}$ totally positive. Denote $\downarrow b\coloneqq\Mid{b}-\diag(s)\Rad{b}$ and $\uparrow b\coloneqq\Mid{b}+\diag(s)\Rad{b}$. Denote by $\geq^*$ the checkerboard order, that is, $u\geq^* v$ iff $\diag(s)u\geq\diag(s)v$. Eventually, the interval vector $[v^1,v^2]^*$ with $v^1 \leq^* v^2$ induced by the checkerboard order is defined as 
$$
[v^1,v^2]^*\coloneqq\diag(s)[\diag(s)v^1,\diag(s)v^2].
$$
Then interval hull of the solution se reads
\begin{itemize}
\item
$\ihull\Ss=[(\uparrow A)^{-1}(\downarrow b),\,(\downarrow A)^{-1}(\uparrow b)]^*$ 
 when $\downarrow b\geq^*0$,
\item
$\ihull\Ss=[(\downarrow A)^{-1}(\downarrow b),\,(\uparrow A)^{-1}(\uparrow b)]^*$ 
 when $\uparrow b\leq^*0$,
\item
$\ihull\Ss=[(\downarrow A)^{-1}(\downarrow b),\,(\downarrow A)^{-1}(\uparrow b)]^*$ 
 when $0\in\ivr{b}$.
\end{itemize}

For an extension of totally positive matrices to the so called sign regular matrices with a prescribed signature; see Garloff et al. \cite{GarAdm2016a}.

Notice that totally positive matrices are componentwisely nonnegative, so all results from Section~\ref{sNonneg} are valid for totally positive matrices, too.

\section{P-matrices}

A square real matrix is a P-matrix if all its principal minors are positive. 
The problem of checking whether a given matrix is a P-matrix is co-NP-hard \cite{Cox1994,KreLak1998}. Fortunately, there are several effectively recognizable sub-classes of P-matrices, such as positive definite matrices, totally positive matrices, (inverse) M-matrices or more generally H-matrices with positive diagonal entries.
By Bia{\l}as and Garloff \cite{BiaGar1984}, an interval matrix $\imace{A}\in\IR^{n\times n}$ is a P-matrix if and only if $\Mid{A}-\diag(z)\Rad{A}\diag(z)$ is a P-matrix for each $z\in\{\pm1\}^n$.

Positive definiteness is easily verifiable for real matrices, but for interval ones it is co-NP-hard \cite{KreLak1998,Roh1994}, so they do not constitute a polynomial sub-class of interval P-matrices. On the other hand, totally positive matrices, M-matrices or H-matrices with positive diagonal are such a sub-class, as we already observed above. The following result shows that as long as the midpoint matrix $\Mid{A}$ of an interval P-matrix is an H-matrix, then $\imace{A}$ itself must be an H-matrix.

\begin{theorem}
Let $\Mid{A}$ be an M-matrix. Then $\imace{A}$ is a P-matrix if an only if it is an H-matrix.
\end{theorem}

\begin{proof}
\quo{If.}
It is obvious. Notice that every matrix in $\imace{A}$ must have positive diagonal.

\quo{Only if.}
Since $\Mid{A}$ is an M-matrix and $\imace{A}$ is regular, the interval matrix $\imace{A}$ must be an H-matrix in view of Theorem~\ref{thmHmatMidReg}. 
\end{proof}

In Hlad\'{\i}k \cite{Hla2017b}, it was shown that an interval matrix $\imace{P}$ with either $\Mid{A}$ or $\Rad{A}$ diagonal is a P-matrix if and only if $\umace{A}$ is a P-matrix. This reduces the problem to just one case, which is however still hard to check in general.

Let us mention one more polynomially decidable subclass of interval P-matrices.
A matrix $A\in\R^{n\times n}$ is a \emph{B-matrix} if
\begin{align*}
\sum_{j=1}^na_{ij}>0 \quad \mbox{and}\quad 
\frac{1}{n}\sum_{j=1}^na_{ij}> a_{ik}\ \forall i\not=k.
\end{align*}
Any B-matrix is a P-matrix; see Pe{\~n}a \cite{Pen2001}.
For an interval matrix $\imace{A}\in\IR^{n\times n}$, B-matrix property is easily checked by adapting the above characterization.

\begin{theorem}
$\imace{A}\in\IR^{n\times n}$ is a B-matrix if and only if
\begin{align*}
\sum_{j=1}^n \unum{a}_{ij}>0 \quad \mbox{and}\quad 
\sum_{j\not=k} \unum{a}_{ij}> (n-1)\onum{a}_{ik}\ \forall i\not=k.
\end{align*}
\end{theorem}

\section{Diagonally interval matrices}

We say that an interval matrix $\imace{A}\in\IR^{n\times n}$ is \emph{diagonally interval} if $\Rad{A}$ is diagonal. These matrices are still intractable from many viewpoints. As shown in Rump \cite{Rum2003}, checking P-matrix property, which is co-NP-hard, can be reduced to checking regularity of an interval matrix $\imace{A}\in\IR^{n\times n}$ with $\Rad{A}=I_n$. Therefore, checking regularity of a diagonally interval matrix is co-NP-hard as well. Similarly, there will be hard many problems related to solving interval linear equations.

On the other hard, regularity turns out to be tractable as long as $\Mid{A}$ is symmetric. Moreover, we can effectively determine all eigenvalues of $\imace{A}$. The following theorem extends the result from Hlad\'{\i}k \cite{Hla2015b}.

\begin{theorem}
Let $\imace{A}\in\IR^{n\times n}$ be diagonally interval and $\Mid{A}$ symmetric. Then $\lambda_i(\smace{A})=[\lambda_i(\umace{A}),\lambda_i(\omace{A})]$ for every $i=1,\dots,n$.
\end{theorem}

\begin{proof}
By the Courant--Fischer theorem we have for every $A\in\imace{A}$
$$
\lambda_i(A)
=\max_{S:\dim(S)=i}\ \min_{x\in S,\,\|x\|=1} x^TAx
\leq\max_{S:\dim(S)=i}\ \min_{x\in S,\,\|x\|=1} x^T\omace{A}x
=\lambda_i(\omace{A}),
$$
and similarly for the lower bound.
\end{proof}

As a simple consequence, we have:

\begin{corollary}
Let $\imace{A}\in\IR^{n\times n}$ be diagonally interval and $\Mid{A}$ symmetric. Then $
\onum{\rho}(\smace{A})=
 \max\{\lambda_1(\omace{A}),-\lambda_n(\umace{A})\}.
$
\end{corollary}

Since the upper bounds for the eigenvalues intervals are attained for the same matrix $\omace{A}$ and analogously for the lower bounds, we get as a consequence simple formula for the range of the determinant provided $\Mid{A}$ is positive semidefinite. This is not the case for a general diagonally interval matrix.

\begin{corollary}
Let $\imace{A}\in\IR^{n\times n}$ be diagonally interval and $\Mid{A}$ symmetric positive semidefinite. Then $\det(\imace{A})=[\det(\umace{A}),\det(\omace{A})]$.
\end{corollary}

In Kosheleva et al.~\cite{KoshKre2005}, it was shown that computing the cube of an interval matrix is an NP-hard problem. Here, we show that it is a polynomial problem provided $\imace{A}\in\IR^{n\times n}$ is diagonally interval. The cube is naturally defined as $\imace{A}^3\coloneqq\{A^3\mmid A\in\imace{A}\}$. It needn't be an interval matrix, so the problem practically is to determine the interval matrix $\ihull\imace{A}^3$.

We will compute the cube entrywise. Let $i,j\in\seznam{n}$ and suppose that $i\not=j$; the case $i=j$ is dealt with analogously. Then the problem is to determine the range of $A^3_{ij}=\sum_{k,\ell}a_{ik}a_{k\ell}a_{\ell j}$ on $a_{kk}\in\inum{a}_{kk}$, $k=1,\dots,n$. This function is linear in $a_{kk}$ for $k\not=i,j$, so we can fix the values of these parameters on the lower or upper bounds, depending on the signs of the corresponding coefficients. Thus the function $A^3_{ij}$ reduces to a quadratic function of variables $a_{ii}$ and $a_{jj}$ only. This can be resolved by brute force by binary search or by utilizing optimality criteria from mathematical programming -- notice that we minimize/maximize quadratic function on a two-dimensional rectangle.

Therefore, we have:

\begin{theorem}
Computing $\ihull\imace{A}^3$ is a polynomial problem for $\imace{A}$ diagonally interval.
\end{theorem}

\section{Nonnegative matrices}\label{sNonneg}

For a (componentwisely) nonnegative matrix $A\in\R^{n\times n}$, the Perron theory says that its spectral radius $\rho(A)$ is attained as the eigenvalue. Let $\imace{A}\in\IR^{n\times n}$. Obviously, it is nonnegative if and only if $\umace{A}\geq0$.
In some situations, however, it is not necessary to assume that all matrices in $\imace{A}$ are nonnegative, but it is sufficient to assume that $\Mid{A}\geq0$.
First, we consider the spectral radius.

\begin{theorem}
We have:
\begin{enumerate}[(i)]
\item
If $\Mid{A}\geq0$, then $\onum{\rho}(\imace{A})=\rho(\omace{A})$.
\item
If $\imace{A}$ is nonnegative, then $\rho(\imace{A})=[\rho(\umace{A}),\rho(\omace{A})]$.
\end{enumerate}
\end{theorem}

\begin{proof}
For every $A\in\imace{A}$, $|A|\leq\Mid{A}+\Rad{A}=\omace{A}$, whence $\rho(A)\leq\rho(\omace{A})$.
If in addition $\umace{A}\geq0$, then $\rho(\umace{A})\leq\rho(A)$ for every $A\in\imace{A}$.
\end{proof}

Analogously, we obtain:

\begin{theorem}
We have:
\begin{enumerate}[(i)]
\item
If $\Mid{A}\geq0$, then $\onum{\lambda}_{\max}(\imace{A})=\lambda_{\max}(\omace{A})$.
\item
If $\imace{A}$ is nonnegative, then $\lambda_{\max}(\imace{A})=[\lambda_{\max}(\umace{A}),\lambda_{\max}(\omace{A})]$.
\end{enumerate}
\end{theorem}

\begin{theorem}\label{thmNonnegSvd}
If $\imace{A}$ is nonnegative, then $\sigma_{\max}(\imace{A})=[\sigma_{\max}(\umace{A}),\sigma_{\max}(\omace{A})]$.
\end{theorem}

Recall that a matrix norm is monotone if $|A|\leq B$ implies $\|A\|\leq\|B\|$. This is satisfied for most of the norms used. For instance, any induced $p$-norm, $\|\cdot\|_{\infty,1}$ norm, Frobenius norm or the Chebyshev norm are monotone.

\begin{theorem}
For every monotone matrix norm we have
\begin{enumerate}[(i)]
\item
If $\Mid{A}\geq0$, then $\onum{\|\imace{A}\|}=\|\omace{A}\|$.
\item
If $\imace{A}$ is nonnegative, then $\|\imace{A}\|=[\|\umace{A}\|,\|\omace{A}\|]$.
\end{enumerate}
\end{theorem}

\begin{proof}
For every $A\in\imace{A}$, we have $|A|\leq\omace{A}$, and therefore $\|A\|\leq\|\omace{A}\|$.
If in addition $\umace{A}\geq0$, then $\|\umace{A}\|\leq \|A\|$ for every $A\in\imace{A}$.
\end{proof}

Nonnegative matrices are also useful for computing high powers of them. Recall that by definition, $\imace{A}^k=\{A^k\mmid A\in\imace{A}\}$. Notice that not every matrix in $[\umace{A}^k,\omace{A}^k]$ is achieved as the $k$th power of some $A\in\imace{A}$, so $\imace{A}^k$ is not an interval matrix.

\begin{theorem}
If $\imace{A}$ is nonnegative, then $\ihull\imace{A}^k=[\umace{A}^k,\omace{A}^k]$.
\end{theorem}

\begin{proof}
Obviously, for every $A\in\imace{A}$, we have $\umace{A}^k\leq A^k\leq \omace{A}^k$. 
\end{proof}

\section{Inverse M-matrices}

A matrix $A\in\R^{n\times n}$ is an \emph{inverse M-matrix} \cite{JohSmi2011} if is nonsingular and $A^{-1}$ is an M-matrix. This represents another easily recognizable sub-class of P-matrices.
Recall that a vertex matrix of  $\imace{A}$ is a matrix $A\in\imace{A}$ such that $a_{ij}\in\{\unum{a}_{ij},\onum{a}_{ij}\}$ for all $i,j$.
Johnson and Smith \cite{JohSmi2002,JohSmi2011} showed that $\imace{A}$ is an inverse M-matrix if and only if all vertex matrices are. This reduces the problem to $2^{n^2}$ real matrices. Neither a polynomial reduction is know, nor NP-hardness was proved. So the computational complexity of checking whether an interval matrix $\imace{A}\in\IR^{n\times n}$ is an inverse M-matrix is an open problem.
It is also worth mentioning the result by Poljak and Rohn \cite{PolRoh1993,Fie2006}, who showed that checking regularity of an interval matrix $[A-ee^T,A+ee^T]$ is co-NP-hard even when $A$ is a symmetric inverse M-matrix.

Since an inverse M-matrix is nonnegative, all results from Section~\ref{sNonneg} are valid in this context, too.

For the componentwise range of inverse matrices, we have the following observation reducing the problem to $2n^2$ real matrices.

\begin{theorem}
If $\imace{A}$ is an inverse M-matrix, then 
\begin{align*}
\min_{A\in\imace{A}}\ A^{-1}
&=\min\left\{(\Mid{A}+\diag(z^i)\Rad{A}\diag(z^j))^{-1}\mmid i,j=1,\dots,n\right\},\\
\max_{A\in\imace{A}}\ A^{-1}
&=\max\left\{(\Mid{A}-\diag(z^i)\Rad{A}\diag(z^j))^{-1}\mmid i,j=1,\dots,n\right\},
\end{align*}
where the minimum is understood componentwisely and $z^i:=(1,\dots,1,-1,1,\dots,1)^T$ has $-1$ in the $i$th entry and $1$ elsewhere.
\end{theorem}

\begin{proof}
The derivative of the inverse is $\frac{\partial (A^{-1})_{ij}}{\partial a_{k\ell}}
=- (A^{-1})_{ik}(A^{-1})_{\ell j}$, or in a matrix form, $\frac{\partial (A^{-1})_{ij}}{\partial A}=- (A^{-T})_{*i}(A^{-T})_{j*}$. It has constant signs, so the minimum value of $(A^{-1})_{ij}$ is attained for the matrix $\Mid{A}+\diag(z^j)\Rad{A}\diag(z^i)$, and analogously the maximum.
\end{proof}

This characterization leads us to the open problem:

\begin{conjecture}
$\imace{A}$ is an inverse M-matrix if and only if $\Mid{A}\pm\diag(z^i)\Rad{A}\diag(z^j))$, $i,j=1,\dots,n$, are inverse M-matrices.
\end{conjecture}

It is also an open question whether interval systems of linear equations $\imace{A}x=\ivr{b}$ can be solved efficiently provided $\imace{A}$ is inverse M-matrix. Anyway, we can state a partial result concerning the interval hull of the solution set.

\begin{theorem}
If $\imace{A}$ is an inverse M-matrix, then $\unum{\ihull\Ss}_i$ is attained for $b:=\Mid{b}+\diag(z^i)\Rad{b}$, and  $\onum{\ihull\Ss}_i$ is attained for $b:=\Mid{b}-\diag(z^i)\Rad{b}$.
\end{theorem}

\begin{proof}
Let $A\in\imace{A}$, $b\in\ivr{b}$ and $x:=A^{-1}b$. Then 
$$
x_i
=A^{-1}_{i*}b
=\sum_{j=1}^n (A^{-1})_{ij}b_j
\geq  (A^{-1})_{ii}\unum{b}_j + \sum_{j\not=i} (A^{-1})_{ij}\onum{b}_j
= A^{-1}_{i*} (\Mid{b}+\diag(z^i)\Rad{b}).
$$
Similarly for the upper bound.
\end{proof}

\begin{theorem}
If $\imace{A}$ is an inverse M-matrix, then $\det(\imace{A})=[\det(A^1),\det(A^2)]$ , where
\begin{align*}
A^1_{ij}=\begin{cases}
\unum{A}_{ii}&\mbox{if }i=j,\\
\onum{A}_{ij}&\mbox{if }i\not=j,
\end{cases}\quad
A^2_{ij}=\begin{cases}
\onum{A}_{ii}&\mbox{if }i=j,\\
\unum{A}_{ij}&\mbox{if }i\not=j.
\end{cases}
\end{align*}
\end{theorem}

\begin{proof}
Similar to the proof of Theorem~\ref{thmMmatDet}. The derivative of the determinant $\det(A)$ is $\det(A)A^{-T}$. The determinant itself is positive, the diagonal of $A^{-T}$ is positive, and its offdiagonal is nonpositive.
\end{proof}

\section{Parametric matrices}

A parametric matrix extends the notion of an interval matrix to a broader class of matrices. A linear parametric matrix is a set of matrices
\begin{align*}
A(p)=\sum_{k=1}^K A^{(k)}p_k,
\end{align*}
where $A^{(1)},\dots,A^{(K)}\in\R^{n\times n}$ are fixed matrices and $p_1,\dots,p_K$ are parameters varying respectively in $\inum{p}_1,\dots,\inum{p}_K\in\IR$. In short, we will denote it as $A(\ivr{p})$.

Since many problems are intractable for standard interval matrices, handling parametric matrices is much more difficult task. 
On the other hand, there are several tractable cases, which we will be concerned with now.

By Hlad\'{\i}k \cite{Hla2017bc}, we have:

\begin{theorem}
$A(\ivr{p})$ is positive definite 
if and only if $A(p)$ is positive definite for each $p$ such that $p_k\in\{\unum{p}_k,\onum{p}_k\}$ $\forall k$. 
\end{theorem}

This reduced the problem to checking positive definiteness of $2^K$ real matrices. Provided $K$ is fixed, we arrived at a polynomial method for checking positive definiteness of $A(\ivr{p})$.

Consider now a parametric system of linear equations
\begin{align*}
A(\ivr{p})x=b(\ivr{p}),
\end{align*}
where $b(p)=\sum_{k=1}^K b^{(k)}p_k$ is a linear parametric right-hand side vector. The corresponding solution set is defined as
\begin{align*}
\Ssp\coloneqq
\{x\in\R^n\mmid \exists p\in\ivr{p}:A(p)x=b(p)\}.
\end{align*}
In contrast to ordinary interval linear systems, characterizing this solution set is a tough problem \cite{AleKre1997,AleKre2003,May2015} even for some  particular linear systems.
Nevertheless, there are some easy-to-handle situations. By Mohsenizadeh et al. \cite{MohKee2014}, under a rank one assumption,  we have a reduction to $2^K$ real systems, which is tractable for fixed number of parameters.

\begin{theorem}
If $\rank(A^{(k)})\leq1$ for every $k=1,\dots,K$, and there are no cross dependencies between the constraint matrix $A(\ivr{p})$ and the right-hand side $b(\ivr{p})$ (i.e., $A^{(k)}\not=0\ \Rightarrow\ b^{(k)}=0$), then the extremal values of $\Ssp$ are attained for $p_k\in\{\unum{p}_k,\onum{p}_k\}$, $k=1,\dots,K$.
\end{theorem}

Another reduction to  $2^K$ real linear systems can be performed based on the result by Popova \cite{Pop2009}. 
\begin{theorem}
If each parameter is involved in one equation only, then $\Ssp$ is described by
\begin{align*}
|A(\Mid{p})x-b(\Mid{p})| \leq \sum_{k=1}^K\Rad{p}_k|A^{(k)}x-b^{(k)}|.
\end{align*}
\end{theorem}

Let $z\in\{\pm1\}^K$ and consider the restriction of $\Ssp$ to the set described by $z_k(A^{(k)}x-b^{(k)})\geq0$, $k=1,\dots,K$. This restricted set has simplified description
\begin{align*}
 A(\Mid{p})x-b(\Mid{p}) &\leq \sum_{k=1}^K\Rad{p}_kz_k(A^{(k)}x-b^{(k)}),\\
-A(\Mid{p})x+b(\Mid{p}) &\leq \sum_{k=1}^K\Rad{p}_kz_k(A^{(k)}x-b^{(k)}),\\
z_k(A^{(k)}x-b^{(k)})&\geq0,\quad  k=1,\dots,K.
\end{align*}
This is a system of linear inequalities, which is efficiently processed via linear programming. Again, we got a reduction to $2^K$, which is a polynomial case provided $K$ is fixed.

\section*{Conclusion}

In this paper, we briefly surveyed interval versions of selected special types of matrices and their useful properties. In particular, we highlighted the properties and characteristics that are efficiently computable even in the interval context. We were motivated by the fact that matrices appearing in applications are not general, but usually have some special structure. Utilizing this special form may in turn radically reduce computational complexity of problems involving the matrices. 



\bibliographystyle{spmpsci}
\bibliography{persp_compl}

\end{document}